\documentclass[11pt,reqno]{amsart}
\usepackage{color}

\usepackage{amsfonts,amscd}
\usepackage{amssymb}
\usepackage{hyperref}
\usepackage{url}
\usepackage[english]{babel}

\theoremstyle{plain}
\newtheorem{theorem}                 {Theorem}      [section]

\newtheorem{corollary}    [theorem]  {Corollary}
\newtheorem{lemma}        [theorem]  {Lemma}
\newtheorem{proposition}  [theorem]  {Proposition}

\theoremstyle{definition}

\newtheorem{definition}   [theorem]  {Definition}

\newtheorem{remark}       [theorem]  {Remark}

\numberwithin{equation}{section}

\def \cn{{\mathbb C}}
\def \C{{\mathbb C}}
\def \hn{{\mathbb H}}

\def \rn{{\mathbb R}}

\def \zn{{\mathbb Z}}

\def \A{\mathcal A}
\def \B{\mathcal B}
\def \E{\mathcal E}
\def \F{\mathcal F}

\def \P{\mathcal P}

\def \Z{\mathcal Z}

\def\nab#1#2{\hbox{$\nabla$\kern -.3em\lower 1.0 ex
		\hbox{$#1$}\kern -.1 em {$#2$}}}

\def \Re{\mathfrak R\mathfrak e}

\def \g{\mathfrak{g}}

\def \GLC#1{\mathbf{GL}_{#1}(\cn)}
\def \glc#1{\mathfrak{gl}_{#1}(\cn)}
\def \GLH#1{\mathbf{GL}_{#1}(\hn)}

\def \SL2{\widetilde{\text{\bf SL}}_{2}(\rn)}

\def \SO#1{\mathbf{SO}(#1)}

\def \U#1{\text{\bf U}(#1)}
\def \u#1{\mathfrak{u}(#1)}

\def \SU#1{\text{\bf SU}(#1)}

\def \Sp#1{\text{\bf Sp}(#1)}
\def \sp#1{\mathfrak{sp}(#1)}

\DeclareMathOperator{\Div}{div} 

\DeclareMathOperator{\trace}{trace}

\numberwithin{equation}{section}
\allowdisplaybreaks

\begin{document}

\title[Explicit harmonic morphisms and $p$ harmonic functions]
{Explicit harmonic morphisms and  $p$-harmonic functions from the complex and quaternionic Grassmannians}

\author{Elsa Ghandour}
\address{Mathematics, Faculty of Science\\
University of Lund\\
Box 118, Lund 221 00\\
Sweden}
\email{Elsa.Ghandour@hotmail.com}

\author{Sigmundur Gudmundsson}
\address{Mathematics, Faculty of Science\\
	University of Lund\\
	Box 118, Lund 221 00\\
	Sweden}
\email{Sigmundur.Gudmundsson@math.lu.se}

\begin{abstract}
We construct explicit complex-valued $p$-harmonic functions and harmonic morphisms on the classical compact symmetric complex and quaternionic Grassmannians.  The ingredients for our construction method are joint eigenfunctions of the classical Laplace-Beltrami and the so called conformality operator. A known duality principle implies that these $p$-harmonic functions and harmonic morphisms also induce such solutions on the Riemannian symmetric non-compact dual spaces.
\end{abstract}




\dedicatory{Corresponding Author: Sigmundur.Gudmundsson@math.lu.se}

\maketitle

\section{Introduction}
\label{section-introduction}

For nearly two centuries mathematicians and physicists have been interested in biharmonic functions.  They appear in several fields such as continuum mechanics, elasticity theory, as well as two-dimensional hydrodynamics problems involving Stokes flows of incompressible Newtonian fluids. Biharmonic functions have a rich and interesting history, a survey of which can be found in the article \cite{Mel}.  The literature on biharmonic functions is vast, but with only very few exceptions the domains are either surfaces or open subsets of flat Euclidean space, see for example \cite{Cad}. The development of the last few  years has changed this situation and can be traced at the regularly updated online bibliography \cite{Gud-p-bib}, maintained by the second author. 
\medskip

It turns out that a natural habitat for complex-valued $p$-harmonic functions and harmonic morphisms are the classical Riemannian symmetric spaces.  They come in pairs consisting of a non-compact $G/K$ and its compact dual companion $U/K$.  It is known that this duality implies that solutions on one of these spaces provides the same on the other.  For this reason we shall here only deal with the compact Riemannian symmetric spaces.

In this paper we manufacture explicit complex-valued proper $p$-harmonic functions and harmonic morphisms on both the complex and the quaternionic Grassmannians
$$\U{n+m}/\U n\times\U m\ \ \text{and}\ \ \Sp{n+m}/\Sp n\times\Sp m.$$
For this we apply two different construction techniques which are presented in Theorem \ref{theorem-rational} and the recent Theorem \ref{theorem-p-harmonic}.  The main ingredients for both these recipes are the common eigenfunctions of the tension field $\tau$ and the conformality operator $\kappa$.  Such eigenfunctions have earlier been constructed on the compact simple Lie groups $\SO n$, $\SU n$ and $\Sp n$ in \cite{Gud-Sak-1}, on the symmetric spaces 
$$\SU n/\SO n,\ \Sp n/\U n,\ \SO {2n}/\U n,\ \SU{2n}/\Sp n$$ 
in \cite{Gud-Sif-Sob-2} and the real Grassmannians $\SO{n+m}/\SO n\times\SO m$ in \cite{Gha-Gud-4}. This means that with this paper we complete the list of all the compact classical Riemannian symmetric spaces.

\section{Eigenfunctions and Eigenfamilies }
\label{section-eigenfunctions}

Let $(M,g)$ be an $m$-dimensional Riemannian manifold and $T^{\cn}M$ be the complexification of the tangent bundle $TM$ of $M$. We extend the metric $g$ to a complex-bilinear form on $T^{\cn}M$.  Then the gradient $\nabla\phi$ of a complex-valued function $\phi:(M,g)\to\cn$ is a section of $T^{\cn}M$.  In this situation, the well-known complex linear {\it Laplace-Beltrami operator} (alt. {\it tension field}) $\tau$ on $(M,g)$ acts locally on $\phi$ as follows
$$
\tau(\phi)=\Div (\nabla \phi)=\sum_{i,j=1}^m\frac{1}{\sqrt{|g|}} \frac{\partial}{\partial x_j}
\left(g^{ij}\, \sqrt{|g|}\, \frac{\partial \phi}{\partial x_i}\right).
$$
For two complex-valued functions $\phi,\psi:(M,g)\to\cn$ we have the following well-known fundamental relation
\begin{equation}\label{equation-basic}
\tau(\phi\, \psi)=\tau(\phi)\,\psi +2\,\kappa(\phi,\psi)+\phi\,\tau(\psi),
\end{equation}
where the complex bilinear {\it conformality operator} $\kappa$ is given by $$\kappa(\phi,\psi)=g(\nabla \phi,\nabla \psi).$$  Locally this satisfies 
$$\kappa(\phi,\psi)=\sum_{i,j=1}^mg^{ij}\cdot\frac{\partial\phi}{\partial x_i}\frac{\partial \psi}{\partial x_j}.$$

\begin{definition}\cite{Gud-Sak-1}\label{definition-eigenfamily}
Let $(M,g)$ be a Riemannian manifold. Then a complex-valued function $\phi:M\to\cn$ is said to be an {\it eigenfunction} if it is eigen both with respect to the Laplace-Beltrami operator $\tau$ and the conformality operator $\kappa$ i.e. there exist complex numbers $\lambda,\mu\in\cn$ such that $$\tau(\phi)=\lambda\cdot\phi\ \ \text{and}\ \ \kappa(\phi,\phi)=\mu\cdot \phi^2.$$	
A set $\E =\{\phi_i:M\to\cn\ |\ i\in I\}$ of complex-valued functions is said to be an {\it eigenfamily} on $M$ if there exist complex numbers $\lambda,\mu\in\cn$ such that for all $\phi,\psi\in\E$ we have 
$$\tau(\phi)=\lambda\cdot\phi\ \ \text{and}\ \ \kappa(\phi,\psi)=\mu\cdot \phi\,\psi.$$ 
\end{definition}

With the following result we show that a given eigenfamily $\E$ can be used to produce a large collection $\P_d(\E)$ of such objects.

\begin{theorem}\label{theorem-polynomials}
Let $(M,g)$ be a Riemannian manifold and the set of complex-valued functions  $$\E=\{\phi_i:M\to\cn\,|\,i=1,2,\dots,n\}$$ 
be a finite eigenfamily i.e. there exist complex numbers $\lambda,\mu\in\cn$ such that for all $\phi,\psi\in\E$ $$\tau(\phi)=\lambda\cdot\phi\ \ \text{and}\ \ \kappa(\phi,\psi)=\mu\cdot\phi\,\psi.$$  
Then the set of complex homogeneous polynomials of degree $d$
$$\P_d(\E)=\{P:M\to\cn\,|\, P\in\cn[\phi_1,\phi_2,\dots,\phi_n],\, P(\alpha\cdot\phi)=\alpha^d\cdot P(\phi),\, \alpha\in\cn\}$$ 
is an eigenfamily on $M$ such that for all $P,Q\in\P_d(\E)$ we have
$$\tau(P)=(d\,\lambda+d(d-1)\,\mu)\cdot P\ \ \text{and}\ \ \kappa(P,Q)=d^2\mu\cdot P\, Q.$$
\end{theorem}

\begin{proof}
The tension field $\tau$ is linear and the conformality operator $\kappa$ is bilinear.  For this reason it is suffcient to prove the result for monomials only. Furthermore, because of the symmetry of $\kappa$ we need only to pick $P$ and $Q$ of the form $P=\phi_1^d$ and $Q=\phi_2^d$.  Then
$$\kappa(P,Q)=\kappa(\phi_1^d,\phi_2^d)=d\cdot\phi_1^{d-1}\cdot\kappa(\phi_1,\phi_2)\cdot d\cdot\phi_2^{d-1}=d^2\mu\cdot P\, Q.$$	
For $\phi\in\P_1(\E)$, we know that $\tau (\phi)=\lambda\cdot \phi$ so the first statement is satisfied for $d=1$.  It then follows by the induction hypothesis that 
\begin{eqnarray*}
\tau(\phi^{d+1})&=&\tau(\phi^d)\,\phi+2\, \kappa(\phi^d,\phi)+\phi^d\,\tau(\phi)\\
&=&(d\,\lambda+d(d-1)\,\mu)\cdot\phi^{d+1}+2\, d\,\mu\cdot \phi^{d+1}+\lambda\cdot\phi^{d+1}\\
&=&(d+1)\,\lambda+(d+1)d\,\mu)\cdot\phi^{d+1}.
\end{eqnarray*}
\end{proof}

\section{Harmonic Morphisms}
\label{section-harmonic-morphisms}

In this section we discuss the much studied harmonic morphisms between Riemannian manifolds. In Theorem  \ref{theorem-rational} we describe how these can be constructed, via eigenfamilies, in the case when the codomain is the standard complex plane.

\begin{definition}\cite{Fug-1},\cite{Ish}
A map $\phi:(M,g)\to (N,h)$ between Riemannian manifolds is
called a {\it harmonic morphism} if, for any harmonic function
$f:U\to\rn$ defined on an open subset $U$ of $N$ with $\phi^{-1}(U)$
non-empty, $f\circ\phi:\phi^{-1}(U)\to\rn$ is a harmonic function.
\end{definition}

The standard reference for the extensive theory of harmonic morphisms is the book \cite{Bai-Woo-book}, but we also recommend the updated online bibliography \cite{Gud-bib}.  The following characterisation of harmonic morphisms between Riemannian manifolds is due to Fuglede and Ishihara, see \cite{Fug-1}, \cite{Ish}. For the definition of horizontal (weak) conformality we refer to \cite{Bai-Woo-book}.

\begin{theorem}\cite{Fug-1}, \cite{Ish}
	A map $\phi:(M,g)\to (N,h)$ between Riemannian manifolds is a harmonic morphism if and only if it is a horizontally (weakly) conformal harmonic map.
\end{theorem}

When the codomain is the standard Euclidean complex plane a function $\phi:(M,g)\to\cn$ is a harmonic morphism i.e. harmonic and horizontally conformal if and only if $\tau(\phi)=0$ and $\kappa(\phi,\phi)=0$.  This explains why $\kappa$ is called the conformality operator.

The following result of Baird and Eells gives the theory of complex-valued harmonic morphisms a strong geometric flavour.  It provides a useful tool for the construction of minimal submanifolds of codimension two.  This is our main motivation for studying these maps.

\begin{theorem}\cite{Bai-Eel}
Let $\phi:(M,g)\to (N^2,h)$ be a horizontally conformal map from a Riemannian manifold to a surface.  Then $\phi$ is harmonic if and only if its fibres are minimal at regular points $\phi$.
\end{theorem}

The next result shows that eigenfamilies can be utilised to manufacture a variety of harmonic morphisms.

\begin{theorem}\cite{Gud-Sak-1}\label{theorem-rational}
Let $(M,g)$ be a Riemannian manifold and 
$$\E =\{\phi_1,\dots ,\phi_n\}$$ 
be an eigenfamily of complex-valued functions on $M$. If $P,Q:\cn^n\to\cn$ are linearily independent homogeneous polynomials of the same positive degree then the quotient
$$\frac{P(\phi_1,\dots ,\phi_n)}{Q(\phi_1,\dots ,\phi_n)}$$ 
is a non-constant harmonic morphism on the open and dense subset
$$\{p\in M| \ Q(\phi_1(p),\dots ,\phi_n(p))\neq 0\}.$$
\end{theorem}

\section{Proper $p$-Harmonic Functions}
\label{section-p-harmonic-functions}

In this section we describe a method for manufacturing  complex-valued proper $p$-harmonic functions on Riemannian manifolds. This method was recently introduced in \cite{Gud-Sob-1}.
\medskip

\begin{definition}\label{definition-proper-p-harmonic}
Let $(M,g)$ be a Riemannian manifold. For a positive integer $p$, the iterated Laplace-Beltrami operator $\tau^p$ is given by
$$\tau^{0} (\phi)=\phi\ \ \text{and}\ \ \tau^p (\phi)=\tau(\tau^{(p-1)}(\phi)).$$	We say that a complex-valued function $\phi:(M,g)\to\cn$ is
\begin{enumerate}
\item[(i)] {\it $p$-harmonic} if $\tau^p (\phi)=0$, and
\item[(ii)] {\it proper $p$-harmonic} if $\tau^p (\phi)=0$ and $\tau^{(p-1)}(\phi)$ does not vanish identically.
\end{enumerate}
\end{definition}

\begin{theorem}\cite{Gud-Sob-1}\label{theorem-p-harmonic}
Let $\phi:(M,g)\to\cn$ be a complex-valued function on a Riemannian manifold and $(\lambda,\mu)\in\cn^2\setminus\{0\}$ be such that the tension field $\tau$ and the conformality operator $\kappa$ satisfy 
$$\tau(\phi)=\lambda\cdot \phi\ \ \text{and}\ \ \kappa(\phi,\phi)=\mu\cdot\phi^2.$$
Then for any positive natural number $p$, the non-vanishing function 
$$\Phi_p:W=\{ x\in M \,\mid\, \phi(x) \not\in (-\infty,0] \}\to\cn$$ with 
$$\Phi_p(x)= 
\begin{cases}
c_1\cdot\log(\phi(x))^{p-1}, 							& \text{if }\; \mu = 0, \; \lambda \not= 0\\[0.2cm]	c_1\cdot\log(\phi(x))^{2p-1}+ c_{2}\cdot\log(\phi(x))^{2p-2}, 								& \text{if }\; \mu \not= 0, \; \lambda = \mu\\[0.2cm]
c_{1}\cdot\phi(x)^{1-\frac\lambda{\mu}}\log(\phi(x))^{p-1} + c_{2}\cdot\log(\phi(x))^{p-1},	& \text{if }\; \mu \not= 0, \; \lambda \not= \mu
\end{cases}
$$ 
is a proper $p$-harmonic function.  Here $c_1,c_2$ are complex coefficients, not both zero.
\end{theorem}

\begin{proof}
A proof of Theorem \ref{theorem-p-harmonic} can be found in \cite{Gud-Sob-1}.
\end{proof}

\section{The General Linear Group $\GLC n$}
\label{section-GLC}

In this section we will now turn our attention to the concrete Riemannian matrix Lie groups embedded as subgroups of the complex general linear group.
\medskip

The group of linear automorphism of $\cn^n$ is the complex general linear group $\GLC n=\{ z\in\cn^{n\times n}\,|\, \det z\neq 0\}$ of invertible $n\times n$ matrices with its standard representation 
$$z\mapsto
\begin{bmatrix}
	z_{11} & \cdots & z_{1n} \\
	\vdots & \ddots & \vdots \\
	z_{n1} & \cdots & z_{nn}
\end{bmatrix}.
$$
Its Lie algebra $\glc n$ of left-invariant vector fields on $\GLC n$ can be identified with $\cn^{n\times n}$ i.e. the complex linear space of $n\times n$ matrices.  We equip $\GLC n$ with its natural left-invariant Riemannian metric $g$ induced by the standard Euclidean inner product $\glc n\times\glc n\to\rn$ on its Lie algebra $\glc n$ satisfying
$$g(Z,W)\mapsto \Re\,\trace\, (Z\cdot \bar W^t).$$ 
For $1\le i,j\le n$, we shall by $E_{ij}$ denote the element of $\rn^{n\times n}$ satisfying
$$(E_{ij})_{kl}=\delta_{ik}\delta_{jl}$$ and by $D_t$ the diagonal
matrices $D_t=E_{tt}.$ For $1\le r<s\le n$, let $X_{rs}$ and
$Y_{rs}$ be the matrices satisfying
$$X_{rs}=\frac 1{\sqrt 2}(E_{rs}+E_{sr}),\ \ Y_{rs}=\frac
1{\sqrt 2}(E_{rs}-E_{sr}).$$
For the real vector space $\glc n$ we then have the canonical orthonormal basis $\B^\cn=\B\cup i\B$, where 
$$\B=\{Y_{rs}, X_{rs}\,|\, 1\le r<s\le n\}\cup\{D_{t}\,|\, t=1,2,\dots,n\}.$$
\vskip .1cm

Let $G$ be a classical Lie subgroup of $\GLC n$ with Lie algebra $\g$ inheriting the induced left-invariant Riemannian metric, which we shall also denote by $g$.  In the cases considered in this paper, $\B_{\g}=\B^\cn\cap\g$ will be an orthornormal basis for the subalgebra $\g$ of $\glc n$.  By employing the Koszul formula for the Levi-Civita connection $\nabla$ on $(G,g)$, we see that for all $Z,W\in\B_{\g}$ we have
\begin{eqnarray*}
	g(\nab ZZ,W)&=&g([W,Z],Z)\\
	&=&\Re\,\trace\, ((WZ-ZW)\bar Z^t)\\
	&=&\Re\,\trace\, (W(Z\bar Z^t-\bar Z^tZ))\\
	&=&0.
\end{eqnarray*}

If $Z\in\g$ is a left-invariant vector field on $G$ and $\phi:U\to\cn$ is a local complex-valued function on $G$ then the $k$-th order derivatives $Z^k(\phi)$ satisfy
\begin{equation*}\label{equation-diff-Z}
	Z^k(\phi)(p)=\frac {d^k}{ds^k}\bigl(\phi(p\cdot\exp(sZ))\bigr)\Big|_{s=0},
\end{equation*}
\begin{equation*}\label{equation-diff-Z-bar}
	Z^k(\bar\phi)(p)=\frac {d^k}{ds^k}\bigl(\overline{\phi(p\cdot\exp(s Z))}\bigr)\Big|_{s=0}.
\end{equation*}
\smallskip 
\noindent
This implies that the tension field $\tau$ and the conformality operator $\kappa$ on $G$ fulfill 
\begin{equation*}\label{equation-tau}
	\tau(\phi)
	=\sum_{Z\in\B_\g}\bigl(Z^2(\phi)-\nab ZZ(\phi)\bigr)
	=\sum_{Z\in\B_\g}Z^2(\phi),
\end{equation*}	
\begin{equation*}\label{equation-kappa}
	\kappa(\phi,\psi)=\sum_{Z\in\B_\g}Z(\phi)\cdot Z(\psi),
\end{equation*}
where $\B_\g$ is the orthonormal basis $\B^\cn\cap\g$ for the Lie algebra $\g$.

\section{The Unitary Group $\U n$}
\label{section-unitary-group}

In this section we construct new eigenfamilies of complex-valued functions on the unitary group 
$$\U n=\{z\in\cn^{n\times n}\,|\,z\cdot\bar z^t=I\}.$$ 
For its standard complex representation $\pi:\U n\to\GLC n$ on $\cn^n$ we use the notation 
$$\pi:z\mapsto
\begin{bmatrix}
	z_{11} & \cdots & z_{1n} \\
	\vdots & \ddots & \vdots \\
	z_{n1} & \cdots & z_{nn}
\end{bmatrix}.
$$
The Lie algebra $\u n$ of $\U n$ consists of the skew-Hermitian matrices i.e. 
$$\u{n}=\{Z\in\cn^{n\times n}|\ Z+\bar Z^t=0\}.$$ 
We equip $\U{n}$ with its standard biinvariant Riemannian metric $g$, induced by the Killing form of its Lie algebra $\u n$, with
$$g(Z,W)=\Re\trace(Z\cdot\bar W^t).$$ 
The canonical orthonormal basis for the Lie algebra $\u n$ is then given by 
$$\B_{\u n}=\{Y_{rs}, iX_{rs}|\ 1\le r<s\le n\}\cup\{iD_t|\ t=1,\dots ,n\}.$$

\begin{lemma}\label{lemma-Un-basic}
Let $z_{j\alpha}:\U n\to\cn$ be the matrix elements of the standard representation of the unitary group $\U n$.  Then the tension field $\tau$ and the conformality operator $\kappa$  satisfy the following relations 
$$
\tau(z_{j\alpha})=-\,n\cdot z_{j\alpha},\ \ 
\kappa(z_{j\alpha}, z_{k\beta})=\,-z_{j\beta}\,z_{k\alpha},
$$
$$
\tau(\bar z_{j\alpha})=-\,n\cdot \bar z_{j\alpha},\ \ 
\kappa(\bar z_{j\alpha}, \bar z_{k\beta})=-\,\bar z_{j\beta}\,\bar z_{k\alpha},
$$
$$
\kappa(z_{j\alpha}, \bar z_{k\beta})=\delta_{jk}\cdot\delta_{\alpha\beta}.
$$
\end{lemma}

\begin{proof}
	The first two relations were already proven in Lemma 5.1 of \cite{Gud-Sak-1}.  The next two follow simply by conjugation and the complex linearity of $\tau$ and $\kappa$. For the last statement we have 	
	\begin{eqnarray*}
		\kappa(z_{j\alpha},\bar z_{k\beta})
		&=&
		\sum_{Z\in\B_{\u n}}Z(z_{j\alpha})\cdot Z(\bar z_{k\beta})\\
		&=&
		\sum_{r<s}e_j\cdot z\cdot Y_{rs}\cdot e_\alpha^t
		\cdot e_\beta\cdot\overline{(Y_{rs}^t)}\cdot\bar z^t\cdot e_k^t\\
		& &
		+\,\sum_{r<s}e_j\cdot z\cdot (iX_{rs})\cdot e_\alpha^t
		\cdot e_\beta\cdot \overline{(iX_{rs}^t)}\cdot \bar z^t\cdot e_k^t\\
		& & 
		+\,\sum_{t}e_j\cdot z\cdot (iD_{t})\cdot e_\alpha^t
		\cdot e_\beta\cdot \overline{(iD_{t}^t)}\cdot \bar z^t\cdot e_k^t\\
		&=&
		e_j\cdot z\cdot\big(\sum_{r<s}Y_{rs}\cdot 
		E_{\alpha\beta}\cdot Y_{rs}^t\big)\cdot\bar z^t\cdot e_k^t\\
		& &
		+\,e_j\cdot z\cdot\big(\sum_{r<s}X_{rs}\cdot 
		E_{\alpha\beta}\cdot X_{rs}^t\big)\cdot\bar z^t\cdot e_k^t\\
		& &
		+\,e_j\cdot z\cdot\big(\sum_{t}D_{t}\cdot 
		E_{\alpha\beta}\cdot D_{t}^t\big)\cdot\bar z^t\cdot e_k^t\\
		&=&
		e_j\cdot z\cdot \delta_{\alpha\beta}\cdot I\cdot\bar z^t\cdot e_k^t\\
		&=&
		\delta_{\alpha\beta}\cdot\delta_{jk}.
	\end{eqnarray*}
\end{proof}

\begin{proposition}\label{proposition-tau-kappa}
Let $z_{j\alpha}:\U n\to\cn$ be the matrix elements of the standard representation of the unitary group $\U n$.  If $j\neq k$, then 
$$\hat\E_{jk}=\{z_{j\alpha}\cdot\bar z_{k\beta}:\U n\to\cn\, |\, 1\le\alpha,\beta\le n\}$$
is an eigenfamily on $\U n$ such that 
$$\tau(\hat\phi)=\,-2\,n\cdot\hat\phi\ \ \text{and}\ \ \kappa(\hat\phi,\hat\psi)=\,-2\cdot\hat\phi\,\hat\psi$$ for all $\hat\phi,\hat\psi\in\hat\E_{jk}$.
\end{proposition}
	
\begin{proof}
Employing Lemma \ref{lemma-Un-basic}, the basic relation (\ref{equation-basic}) and the fact that $j\neq k$ we yield the following
\begin{eqnarray*}
	\tau\bigl( z_{j\alpha}\, \bar z_{k\beta}\bigr)
	&=&\tau(z_{j\alpha})\cdot \bar z_{k\beta}+2\cdot\kappa(z_{j\alpha},\bar z_{k\beta})+z_{j\alpha}\cdot \tau (\bar z_{k\beta})\\
	&=&-2\,n\cdot z_{j\alpha}\, \bar z_{k\beta}
\end{eqnarray*}	
and
\begin{eqnarray*}
\kappa\big(z_{j\alpha}\,\bar z_{k\beta},z_{j\gamma}\, \bar z_{k\delta}\big)
&=&z_{j\alpha}\, z_{j\gamma}\cdot\kappa\big(\bar z_{k\beta}, \bar z_{k\delta}\big)+z_{j\alpha}\,\bar z_{k\delta}\cdot\kappa\big(\bar z_{k\beta},z_{j\gamma}\, \big)\\
&&\ +\,\bar z_{k\beta}\,\bar z_{k\delta}\cdot \kappa\big(z_{j\alpha},z_{j\gamma}\big)+
\bar z_{k\beta}\,z_{j\gamma}\cdot
\kappa\big(z_{j\alpha},\bar z_{k\delta}\big)\\
&=&-\,2\cdot 
(z_{j\alpha}\, \bar z_{k\beta})
(z_{j\gamma}\, \bar z_{k\delta}).
	\end{eqnarray*}
\end{proof}

\section{Lifting Properties}

We shall now present an interesting connection between the theory of harmonic morphisms and complex-valued $p$-harmonic functions.

\begin{proposition}\label{proposition-lift}
Let $\pi:(\hat M,\hat g)\to (M,g)$ be a Riemannian submersion between Riemannian manifolds. Further let $f:(M,g)\to\C$ be a smooth function and $\hat f:(\hat M,\hat g)\to\C$ be the composition $\hat f=f\circ\pi$. Then the tension fields $\tau$ and $\hat\tau$  satisfy
$$\tau(f)\circ\pi=\hat\tau(\hat f)\ \ \text{and}\ \ \tau^p(f)\circ\pi=\hat\tau^{p}(\hat f)$$
for all positive integers $p\ge 2$.
\end{proposition}

\begin{proof}
The Riemannian submersion $\pi:(\hat M,\hat g)\to (M,g)$ is a harmonic morphism i.e. a horizontally conformal, harmonic map with constant dilation $\lambda=1$. Hence the well-known composition law for the tension field gives
\begin{eqnarray*}
\hat\tau(\hat f)
&=&\hat\tau(f\circ\pi)\\
&=&\text{trace}\nabla df(d\pi,d\pi)+ df(\hat\tau(\pi))\\
&=&\tau(f)\circ\pi+df(\hat\tau(\pi))\\
&=&\tau(f)\circ\pi.
\end{eqnarray*}
The rest follows by induction.
\end{proof}

In the sequel, we shall apply the following immediate consequence of Proposition \ref{proposition-lift}.

\begin{corollary}\label{corollary-lift-proper-p-harmonic}
Let $\pi:(\hat M,\hat g)\to (M,g)$ be a Riemannian submersion.  Further let $\phi:(M,g)\to\C$ be a smooth function and $\hat \phi:(\hat M,\hat g)\to\C$ be the composition $\hat \phi=\phi\circ\pi$. Then the following statements are equivalent
	\begin{enumerate}
		\item[(a)] $\hat \phi:(\hat M,\hat g)\to\C$ is proper $p$-harmonic,
		\item[(b)] $\phi:(M,g)\to\C$ is proper $p$-harmonic.
	\end{enumerate}
\end{corollary}

\section{The Complex Grassmannians}
\label{section-complex-grassmann}

In this section we construct explicit eigenfamilies on the complex Grassmannians $G_m(\cn^{m+n})=\U{m+n}/\U {m}\times\U{n}$ of $m$-dimensional complex subspaces of $\cn^{m+n}$. Here we make use of the fact that the natural projection 
$$\pi:\U{m+n}\to\U{m+n}/\U m\times\U n$$ 
is a Riemannian submersion with totally geodesic fibres and hence a harmonic morphism.
\medskip

For $1\le j< k\le m+n$, we now define the complex-valued functions $\hat \phi_{jk}:\U{m+n}\to\cn$ on the unitary group by
$$\hat\phi_{jk}(z)=\sum_{t=1}^{m}z_{jt}\, \bar z_{k t}.$$

\begin{lemma}\label{lemma-phi-phi-complex}
The tension field $\hat\tau$ and the conformality operator $\hat\kappa$ on the unitary group $\U{m+n}$ satisfy
$$\hat\tau(\hat\phi_{jk})
=-\,2\,(m+n)\cdot\hat\phi_{jk}\ \ \text{and}\ \ 
\hat\kappa(\hat\phi_{jk},\hat\phi_{lm})
=-\,2\,\hat\phi_{jk}\,\hat\phi_{lm}.$$
\end{lemma}

\begin{proof}
The statement follows from an easy calculation employing Lemma \ref{lemma-Un-basic}.
\end{proof}

The functions $\hat\phi_{jk}$ are all $\U{m}\times \U{n}$-invariant and hence they induce functions 
$\phi_{jk}:\U{m+n}/\U m\times\U n\to\cn$
on the compact quotient space.

\begin{theorem}
For a fixed natural number $1\le \alpha < m+n$, the set $$\E_\alpha=\{\phi_{jk}:G_m(\cn^{m+n})\to\cn\,|\,1\le j\le\alpha<k\le m+n \}$$
is an eigenfamily on the complex Grassmannian $G_m(\cn^{m+n})$ such that  the tension field $\tau$ and the conformality operator $\kappa$ satisfy
$$\tau(\phi)=-\,2(m+n)\cdot\phi\ \ \text{and}\ \ \kappa(\phi,\psi)=-\,2\cdot \phi\,\psi$$
for all $\phi,\psi\in\E_\alpha$.
\end{theorem}

\begin{proof}
The statement follows immediately from Lemma \ref{lemma-phi-phi-complex} and Corollary \ref{corollary-lift-proper-p-harmonic}.
\end{proof}

\begin{remark}\label{remark-complex}
For $t>2$ and positive integers $n,n_1,n_2,\dots,n_t\in\zn^+$ with  
$n=n_1+n_2+\cdots +n_t$, let $\F_\cn(n_1,\dots,n_t)$ denote the homogeneous complex flag manifold with 
$$\F_\cn(n_1,\dots,n_t)=\U{n}/\U {n_1}\times\U{n_2}\times\cdots\times\U{n_t},$$ 
Then the standard Riemannian metric on the unitary  group $\U{n}$ induces a natural metric on the complex homogeneous flag manifolds and for these we have the Riemannian fibrations
$$\U{n}\to\F_\cn(n_1,\dots,n_t)\to G_{n_s}(\cn^{n}).$$ 
Here $G_{n_s}(\cn^{n})$ is the complex Grassmannian of $n_s$-dimensional complex subspaces of $\cn^{n}$, for each  $1\le s\le t$. Then a generic element $z\in\U n$ can be written of the form
$$z=[\,z_1\,|\,z_2\,|\,\cdots\, |\,z_t\,],$$
where each $z_s$ is an $n\times n_s$ submatrix of $z$. Following the above we can now, for each block, construct a collection $\hat\E_s$ of $\U {n_s}$-invariant complex-valued eigenfunctions on the unitary group $\U n$ such that for all $\hat\phi_s\in\hat\E_s$
$$\tau (\hat\phi_s)=-\,2\,n\cdot\hat\phi_s\ \ \text{and}\ \ \kappa(\hat\phi_s,\hat\phi_s)=-\,2\cdot\hat\phi_s^{\, 2}.$$
According to Theorem \ref{theorem-p-harmonic},  each function 
$$\hat\Phi_{p,s}(z)=(c_{1,s}+c_{2,s}\cdot\hat\phi_s(z)^{1-n})\cdot\log(\hat\phi_s(z))^{p-1}$$
is proper $p$-harmonic on an open and dense subset of $\U n$.  The sum 
$$\hat\Phi_p=\sum_{s=1}^t\hat\Phi_{p,s}$$ constitutes a multi-dimensional family $\hat\A_p$ of $\U{n_1}\times\cdots\times\U{n_t}$-invariant proper $p$-harmonic functions on an open dense subset of $\U n$.  Furthermore, each element $\hat\Phi_p\in\hat\A_p$ induces a proper $p$-harmonic function $\Phi_p$ defined on an open and dense subset of the complex flag manifold 
$$\F_\cn (n_1,\dots ,n_t)=\U n/\U{n_1}\times\cdots\times\U{n_t},$$ 
which does not descend onto any of the complex Grassmannians since $t> 2 $.
\end{remark}

\section{The Quaternionic Unitary Group $\Sp n$}
\label{section-quaternionic-unitary-group}

In this section we consider the quaternionic unitary group $\Sp n$.  Its standard complex representation $\pi:\Sp n\to\U {2n}$ on $\cn^{2n}$ is given by
$$
\pi:(z+jw)\mapsto g=\begin{bmatrix}
	z_{11}    &  \dots & z_{1n}   & w_{11}    & \dots  & w_{1n}\\
	\vdots    & \ddots & \vdots      & \vdots    & \ddots & \vdots  \\
	z_{n1} & \dots  & z_{nn} & w_{n1} & \dots  & w_{nn}\\
	-\bar w_{11} &  \dots   & -\bar w_{1n}   & \bar z_{11}  & \dots  & \bar z_{1n}\\
	\vdots       & \ddots   & \vdots      & \vdots    & \ddots & \vdots  \\
	-\bar w_{n1} & \dots & -\bar w_{nn} & \bar z_{n1} & \dots  & \bar z_{nn}\\
\end{bmatrix}.
$$
The Lie algebra $\sp n$ of $\Sp n$ satisfies
$$\sp{n}=\{\begin{pmatrix} Z & W
	\\ -\bar W & \bar Z\end{pmatrix}\in\cn^{2n\times 2n}
\ |\ Z^*+Z=0,\ W^t-W=0\}$$ and for this we have the standard
orthonormal basis 
\begin{eqnarray*}
	\B&=&
	\Big\{\frac{1}{\sqrt{2}}
	\begin{bmatrix}
		0 & iX_{rs}\\
		iX_{rs}& 0
	\end{bmatrix}, 
	\frac{1}{\sqrt{2}}
	\begin{bmatrix}
		0 & X_{rs}\\
		-X_{rs}& 0
	\end{bmatrix}, 
	\frac{1}{\sqrt{2}}
	\begin{bmatrix}
		iX_{rs}& 0\\
		0 & -iX_{rs}
	\end{bmatrix},\\
	& &
	\qquad
	\frac{ 1}{\sqrt{2}}
	\begin{bmatrix}
		Y_{rs}& 0\\
		0 & Y_{rs}
	\end{bmatrix},
	\frac{1}{\sqrt{2}}
	\begin{bmatrix}
		0 & D_t\\
		-D_t & 0
	\end{bmatrix}, 
	\frac{1}{\sqrt{2}}
	\begin{bmatrix}
		i D_t & 0\\
		0 & -i D_t
	\end{bmatrix},\\
	& &
	\qquad\qquad
	\frac{1}{\sqrt{2}}
	\begin{bmatrix}
		0 & iD_t\\
		i D_t & 0
	\end{bmatrix}
	\Big|\, 1\leq r<s\leq n,\ 1 \leq t\leq n\Big\}\,.
\end{eqnarray*}

\begin{lemma}\label{lemma-Spn-basic}
For $1\le j,\alpha\le n$, let $z_{j\alpha}, w_{j\alpha}:\Sp n\to\cn$ be the complex-valued matrix elements of the standard representation of the quaternionic unitary group $\Sp n$. Then the tension field $\tau$ and the conformality operator $\kappa$ on $\Sp n$ satisfy the following relations
$$
\tau(z_{j\alpha})= -\,\tfrac{2n+1}2\cdot z_{j\alpha},\ \ \tau(w_{j\alpha})= -\,\tfrac{2n+1}2\cdot w_{j\alpha},
$$
$$
\kappa(z_{j\alpha},z_{k\beta})
=-\,\tfrac 12\cdot z_{j\beta}\,z_{k\alpha},\ \
\kappa(w_{j\alpha},w_{k\beta})
=-\,\tfrac 12\cdot w_{j\beta}\,w_{k\alpha},
$$	
$$
\kappa(z_{j\alpha},w_{k\beta})
=-\,\tfrac 12\cdot w_{j\beta}\,z_{k\alpha},$$
$$\tau(\bar z_{j\alpha})= -\,\tfrac{2n+1}2\cdot \bar z_{j\alpha},\ \ \tau(\bar w_{j\alpha})= -\,\tfrac{2n+1}2\cdot \bar w_{j\alpha},
$$
$$
\kappa(\bar z_{j\alpha},\bar z_{k\beta})
=-\,\tfrac 12\cdot\bar z_{j\beta}\,\bar z_{k\alpha},\ \
\kappa(\bar w_{j\alpha},\bar w_{k\beta})
=-\,\tfrac 12\cdot\bar w_{j\beta}\,\bar w_{k\alpha},
$$	
$$
\kappa(\bar z_{j\alpha},\bar w_{k\beta})
=-\,\tfrac 12\cdot\bar  w_{j\beta}\,\bar z_{k\alpha},
$$
$$
\kappa(z_{j\alpha},\bar z_{k\beta})
=\tfrac 12\cdot(w_{j\beta}\,\bar  w_{k\alpha}+\delta_{jk}\delta_{\alpha\beta}),\ \ 
\kappa( z_{j\alpha},\bar w_{k\beta})
=-\tfrac 12 z_{j\beta}\,\bar w_{k\alpha},
$$
$$
\kappa( w_{j\alpha},\bar z_{k\beta})
= -\tfrac 12w_{j\beta}\,\bar z_{k\alpha},
\ \ \kappa( w_{j\alpha},\bar w_{k\beta})
=\tfrac 12\cdot(z_{j\beta}\,\bar z_{k\alpha}+\delta_{jk}\delta_{\alpha\beta}).
$$
\end{lemma}

\begin{proof}
The result can be obtained by exactly the same technique as that of Lemma \ref{lemma-Un-basic}.
\end{proof}

\begin{theorem}
For $1\le j<k\le n$, let $z_{j\alpha},w_{j\alpha}:\Sp n\to\cn$ be the matrix elements of the standard representation of $\Sp n$ and define the complex-valued  $\hat\phi_{jk}^\alpha:\Sp n\to\cn$ with  
$$\hat\phi_{jk}^\alpha
=z_{j\alpha}\,\bar z_{k\alpha}
+w_{j\alpha}\,\bar w_{k\alpha}.$$
Then 
$\hat\E_{jk}=\{\hat\phi_{jk}^\alpha:\Sp n\to\cn\,|\, 1\le\alpha\le n\}$
is an eigenfamily on $\Sp n$ such that the tension field $\hat\tau$ and the conformality operator $\hat\kappa$ satisfy 
$$\hat\tau(\hat\phi)=-2n\cdot \hat\phi\ \ \text{and}\ \ 
\hat\kappa(\hat\phi,\hat\psi)
=-\,\hat\phi\,\hat\psi,$$
for all $\hat\phi,\hat\psi\in\hat\E_{jk}$.
\end{theorem}

\begin{proof}
The statement follows from a standard computation applying Lemma \ref{lemma-Spn-basic} and the basic relation (\ref{equation-basic}).
\end{proof}

\section{The Quaternionic Grassmannians}
\label{section-quaternionic-grassmann}

In this section we construct explicit eigenfamilies on the quaternionic Grassmannians $G_m(\hn^{m+n})=\Sp{m+n}/\Sp {m}\times\Sp{n}$ of $m$-dimensional quaternionic subspaces of $\hn^{m+n}$. Here we make use of the fact that the natural projection 
$$\pi:\Sp{m+n}\to\Sp{m+n}/\Sp m\times\Sp n$$ 
is a Riemannian submersion with totally geodesic fibres and hence a harmonic morphism.
\medskip

For $1\le j< \alpha\le m+n$, we now define the complex-valued functions $\hat \phi_{j\alpha}:\Sp{m+n}\to\cn$ on the unitary group by
$$\hat\phi_{j\alpha}(z)=\sum_{t=1}^{m}(z_{jt}\,\bar z_{\alpha t}+w_{jt}\,\bar w_{\alpha t}).$$

\begin{lemma}\label{lemma-phi-phi-quaternionic}
The tension field $\hat\tau$ and the conformality operator $\hat\kappa$ on the quaternionic unitary group $\Sp{m+n}$ satisfy
$$\hat\tau(\hat\phi_{j\alpha})=-\,2\,(m+n)\cdot\hat\phi_{j\alpha}\ \ \text{and}\ \ \hat\kappa(\hat\phi_{j\alpha},\hat\phi_{k\alpha})
=-\,\hat\phi_{j\alpha}\,\hat\phi_{k\alpha},$$
where $j,k\neq\alpha$.
\end{lemma}

\begin{proof}
The statement follows from an easy calculation employing Lemma \ref{lemma-Spn-basic}.
\end{proof}

The functions $\hat\phi_{j\alpha}$ are all $\Sp{m}\times \Sp{n}$-invariant and hence they induce functions 
$\phi_{j\alpha}:\Sp{m+n}/\Sp m\times\Sp n\to\cn$
on the compact quotient space.

\begin{theorem}
For a fixed natural number $1\le r < m+n$, the set $$\E_r=\{\phi_{j\alpha}:G_m(\hn^{m+n})\to\cn\,|\,1\le j\le r<\alpha\le m+n \}$$
is an eigenfamily on the quaternionic Grassmannian $G_m(\hn^{m+n})$ such that  the tension field $\tau$ and the conformality operator $\kappa$ satisfy
$$\tau(\phi)=-\,2(m+n)\cdot\phi\ \ \text{and}\ \ \kappa(\phi,\psi)=-\,\phi\,\psi$$
for all $\phi,\psi\in\E_r$.
\end{theorem}

\begin{proof}
The statement follows immediately from Lemma \ref{lemma-phi-phi-quaternionic} and Corollary \ref{corollary-lift-proper-p-harmonic}.
\end{proof}

\begin{remark}\label{remark-quaternionic}
Employing the same procedure as in Remark \ref{remark-complex}, it is easy to construct proper $p$-harmonic function defined on open and dense subsets of the quaternionic flag manifold 
$$\F_\hn (n_1,\dots ,n_t)=\Sp n/\Sp{n_1}\times\cdots\times\Sp{n_t},$$ 
which does not descend onto any of the quaternionic  Grassmannians if $t\ge 3$.
\end{remark}

\section{The Unitary Group $\U{m+n}$ Revisited}
\label{section-unitary}

Here we construct new eigenfamilies on the unitary group $\U {m+n}$ as the compact subgroup 
$$\U {m+n}=\{z\in\GLC{m+n}|\ z\cdot \bar z^t=I_{m+n}\}$$
of the complex general linear group $\GLC{m+n}$ of invertible $(m+n)\times (m+n)$ matrices with standard matrix representation
$$
z=\begin{bmatrix}
	z_{11} & z_{12} & \cdots & z_{1,m+n}\\
	z_{21} & z_{22} & \cdots & z_{2,m+n}\\
	\vdots & \vdots & \ddots & \vdots \\
	z_{m+n,1} & z_{m+n,2} & \cdots & z_{m+n,m+n}
\end{bmatrix}.
$$

The Lie algebra $\u{m+n}$ of left-invariant vector fields on $\U{m+n}$, can be identified with skew-Hermitian matrices in $\cn^{(m+n)\times (m+n)}$, which we equip with the standard left-invariant Riemannian metric $g$ such that for all $Z,W\in\u{m+n}$ we have $$g(Z,W)=\Re\trace (Z\cdot \bar W^t).$$

For this we have the following fundamental ingredient for our recipe. 

\begin{lemma}\cite{Gud-Sak-1}\label{lemma-Un}
For $1\le j,\alpha\le m+n$, let $z_{j\alpha}:\U {m+n}\to\cn$ be the complex-valued matrix elements of the standard representation of $\U {m+n}$. Then the tension field $\hat\tau$ and the conformality operator $\hat\kappa$ on $\U{m+n}$ satisfy the following relations
\begin{equation*}
\hat\tau(z_{j\alpha})= -\,(m+n)\cdot z_{j\alpha}\ \ \text{and}\ \ \hat\kappa(z_{j\alpha},z_{k\beta})= -\,z_{j\beta}\, z_{k\alpha}.
\end{equation*}
\end{lemma}
\medskip

Let $\Pi_{m,n}$ denote the set of permutations $\pi=(r_1,r_2,\dots ,r_m)$ such that $1\le r_1<r_2<\dots <r_m\le m+n$ and $\hat\phi_\pi:\U{m+n}\to\cn$ be the determinant of following $m\times m$ submatrix of $z$  associated with $\pi$  
$$\hat\phi_\pi(z)=\det
\begin{bmatrix}
z_{r_1,1}  & \dots  & z_{r_1,m}\\
\vdots     & \ddots & \vdots   \\
z_{r_m,1}  & \dots & z_{r_m,m}
\end{bmatrix}.$$
By ${\bf P}(m+n,m)$ we denote the number of such permutations.
\medskip

We are now ready to present the main results of this section.  This provides a new collection of eigenfamilies $\hat\E_{m,n}$  of complex-valued functions on the unitary group $\U{m+n}$.

\begin{theorem}\label{theorem-main-Un}
The set $\hat\E_m=\{\hat\phi_\pi\,|\, \pi\in\Pi_{m,n}\}$ of complex-valued functions is an eigenfamily on the unitary group $\U{m+n}$, such that the tension field $\hat\tau$ and the conformality operator $\hat\kappa$ satisfy
$$\tau (\hat\phi)=-\,m(n+1)\cdot\hat\phi\ \ \text{and}\ \ \kappa(\hat\phi,\hat\psi)=-\,m\cdot\hat\phi\cdot\hat\psi,$$ 
for all $\hat\phi,\hat\psi\in\hat\E_m$.
\end{theorem}

\begin{proof}
	Let $\hat\phi$ and $\hat\psi$ be two elements of $\hat\E_m$ of the form
	$$
	\hat\phi(z)=\det
	\begin{bmatrix}
		z_{r_1,1}  & \dots  & z_{r_1,m}\\
		\vdots     & \ddots & \vdots   \\
		z_{r_m,1}  & \dots & z_{r_m,m}
	\end{bmatrix},\ \ 
	\hat\psi(z)=\det
	\begin{bmatrix}
		z_{s_1,1}  & \dots  & z_{s_1,m}\\
		\vdots     & \ddots & \vdots   \\
		z_{s_m,1}  & \dots & z_{s_m,m}
	\end{bmatrix}.
	$$
	Using the Levi-Civita symbol $\varepsilon_{i_1\dots i_m}$ we have the standard expression for the determinants 
	$$
	\hat\phi(z)=\sum_{i_1\dots i_m=r_1}^{r_m}\varepsilon_{i_1\dots i_m}\cdot z_{i_1,1}\dots z_{i_m,m}$$
and
$$	
	\hat{\psi}(z)=\sum_{j_1,\dots, j_m=s_1}^{s_m}\varepsilon_{j_1\dots j_m}\cdot z_{j_1,1}\dots z_{j_m,m}\,.
	$$
	Then applying the tension field to the function $\hat\phi$ we yield
	\begin{eqnarray*}
		\tau(\hat{\phi})&=&\sum\limits_{i_1\dots i_m=r_1}^{r_m}\varepsilon_{i_1\dots i_m}\cdot \\
		& &
		\big\{\tau(z_{i_1,1})\cdot z_{i_2,2}\dots z_{i_m,m}+\cdots  +z_{i_1,1}\cdots z_{i_{m-1}m-1}\cdot\tau(z_{i_m,m})\\
		& &\qquad
		+2\cdot [\kappa(z_{i_1,1},z_{i_2,2})\cdot z_{i_3,3}\cdots z_{i_m,m}+\cdots + \\
		& &\qquad\qquad\qquad\, z_{i_1,1}\cdots z_{i_{m-2},{m-2}}
		\cdot\kappa(z_{i_{m-1},{m-1}},z_{i_m,m})]\big\}.
	\end{eqnarray*}
	By now using the relations
	\begin{equation*}
		\hat\tau(z_{j\alpha})= \lambda\cdot z_{j\alpha}\ \ \text{and}\ \ \hat\kappa(z_{j\alpha},z_{k\beta})= \mu\cdot  z_{j\beta}z_{k\alpha}\,,
	\end{equation*}
	we obtain
	\begin{eqnarray*}
		\tau(\hat{\phi})&=&\lambda\cdot m\cdot\hat{\phi}\\
		& &+\,2\,\mu\cdot\sum\limits_{i_1\dots i_m=r_1}^{r_m}\varepsilon_{i_1\dots i_m}\cdot\left\{z_{i_1,2}\cdot z_{i_2,1}\cdot z_{i_3,3}\cdots z_{i_m,m}+\cdots+\right.\\
		& &\qquad\qquad\qquad\qquad z_{i_1,1}\dots z_{i_{m-2},{m-2}}\cdot z_{i_{m-1},{m}}\cdot z_{i_m,{m-1}}\left.\right\}.
	\end{eqnarray*}
	We note that
	$\varepsilon_{i_1\dots i_l\dots i_k\dots i_m}=-\,\varepsilon_{i_1\dots i_k\dots i_l\dots i_m}$, for $1\leq l,k\leq m$, so 
	\begin{eqnarray*}
		\tau(\hat{\phi})&=&-\,(m+n)\cdot m\cdot\hat{\phi}\\
		& &+\,2\sum\limits_{i_1\dots i_m=r_1}^{r_m}\varepsilon_{i_2 i_1\dots i_m}\cdot z_{i_2,1}\cdot z_{i_1,2}\cdot z_{i_3,3}\cdots z_{i_m,m}+\cdots+\\
		& &\qquad\varepsilon_{i_1\dots i_{m-2} i_{m} i_{m-1}}\cdot z_{i_{1},1}\dots z_{i_{m-2},{m-2}}\cdot z_{i_m,{m-1}}\cdot z_{i_{m-1},{m}}\left.\right\}\,
	\end{eqnarray*}
	and now by just changing the indices we yield
	
	$$\tau(\hat{\phi})=-\,(m+n)\cdot m\cdot \hat{\phi}+2\cdot\binom m2\cdot \hat{\phi}\\
	=-\,m\cdot (n+1)\cdot\hat{\phi}.
	$$
	\medskip
	
	The conformality operator satisfies
	\begin{eqnarray*}
		\kappa(\hat{\phi},\hat{\psi})
		&=&\sum\limits_{i_1,\dots,i_m=r_1}^{r_m}\sum\limits_{j_1,\dots,j_m=s_1}^{s_m}\varepsilon_{i_1\dots i_m}\varepsilon_{j_1\dots j_m}\cdot\\
		& &\qquad\{\kappa(z_{i_1,1},z_{j_1,1})\cdot z_{i_2,2}\dots z_{i_m,m}\dots z_{j_2,2}\dots z_{j_m,m}\\
		& &\qquad\qquad +\,\kappa(z_{i_1,1},z_{j_2,2})\cdot z_{i_2,2}\dots z_{i_m,m}z_{j_1,1}z_{j_3,3}\dots z_{j_m,m}\\
		& &\qquad\qquad\qquad\qquad\qquad\vdots\\
		& &\qquad +\,\kappa(z_{i_m,m},z_{j_m,m})\cdot z_{i_1,1}\dots z_{i_{m-1},{m-1}}z_{j_1,1}\cdots z_{j_{m-1},{m-1}}\} \\
		&=&-\sum\limits_{i_1,\dots,i_m=r_1}^{r_m}\sum\limits_{j_1,\dots,j_m=s_1}^{s_m}\varepsilon_{i_1\dots i_m}\varepsilon_{j_1\dots j_m}\cdot\\
		& &\qquad\qquad\{z_{i_1,1}z_{i_2,2}\dots z_{i_m,m} z_{j_1,1} z_{j_2,2}\dots z_{j_m,m}\\
		& &\qquad\qquad +\, z_{i_1,2}z_{i_2,2}\dots z_{i_m,m}z_{j_1,1} z_{j_2,1} z_{j_3,3}\dots z_{j_m,m}\\	
		& &\qquad\qquad\qquad\qquad\qquad\qquad\qquad\vdots\\
		& &\qquad\qquad +z_{i_1,1}\dots z_{i_{m-1},{m-1}}z_{i_m,m-1} z_{j_1,1}\cdots z_{j_{m-1},{m}}z_{j_m,m}\\
		& &\qquad\qquad +z_{i_1,1}\dots z_{i_{m-1},{m-1}}z_{i_m,m} z_{j_1,1}\cdots z_{j_{m-1},{m-1}}z_{j_m,m}\}\, .
	\end{eqnarray*}
	The term  $z_{i_1,1}\dots z_{i_m,m}z_{j_1,1}\dots z_{j_m,m}$ appears $m$ times in the above sum, so
	\begin{eqnarray*}
		&&\kappa(\hat{\phi},\hat{\psi})\\ 
		&=&-m\cdot\hat{\phi}\cdot\hat{\psi}\\
		&&-\sum\limits_{i_1,\dots,i_m=r_1}^{r_m}\sum\limits_{j_1,\dots,j_m=s_1}^{s_m}\varepsilon_{i_1\dots i_m}\varepsilon_{j_1\dots j_m}\cdot\\ &&\qquad\qquad\qquad\qquad  z_{i_1,2}z_{i_2,2}\dots z_{i_m,m}z_{j_1,1} z_{j_2,1} z_{j_3,3}\dots z_{j_m,m}\\
		& &\qquad\qquad\qquad\qquad\qquad\vdots\\
		& &- \sum\limits_{i_1,\dots,i_m=r_1}^{r_m}\sum\limits_{j_1,\dots,j_m=s_1}^{s_m}\varepsilon_{i_1\dots i_m}\varepsilon_{j_1\dots j_m}\cdot\\
		&&\qquad\qquad\qquad\qquad z_{i_1,1}\dots z_{i_{m-1},{m-1}}z_{i_m,m-1} z_{j_1,1}\cdots z_{j_{m-1},{m}}z_{j_m,m}.
	\end{eqnarray*}
	But each term in the last equation containing a double sum vanishes. Because of the symmetry, it is enough to show this, for instance, for the term
	\begin{equation}\label{vanishing terms}
		\sum\limits_{i_1,\dots,i_m=r_1}^{r_m}\sum\limits_{j_1,\dots,j_m=s_1}^{s_m}\varepsilon_{i_1\dots i_m}\varepsilon_{j_1\dots j_m}\cdot z_{i_1,2}  z_{i_2,2}\dots z_{i_m,m}z_{j_1,1} z_{j_2,1} z_{j_3,3}\dots z_{j_m,m}.
	\end{equation}
	We have the following property
	\begin{equation}\label{det-product-epsilon}
		\varepsilon_{i_1i_2\dots i_m}\varepsilon_{j_1j_2\dots j_m}=\det
		\begin{bmatrix}
			\delta_{i_1j_1}& \delta_{i_1j_2}&\dots&\delta_{i_1j_m}\\
			\delta_{i_2j_1}& \delta_{i_2j_2}&\dots&\delta_{i_2j_m}\\
			\vdots&\vdots&\ddots&\vdots\\
			\delta_{i_mj_1}&\delta_{i_mj_2}&\dots&\delta_{i_mj_m}
		\end{bmatrix}.
	\end{equation}
	Equation \eqref{vanishing terms} can be obtained by a simple induction.
	For $m=2$ we have the following	
	\begin{eqnarray*}
		&&
		\sum\limits_{i_1,i_2=r_1}^{r_2}\sum\limits_{j_1,j_2=s_1}^{s_2}\varepsilon_{i_1 i_2}\varepsilon_{j_1 j_2}\cdot z_{i_1,2}  z_{i_2,2} z_{j_1,1} z_{j_2,1}\\
		&=& \sum\limits_{i_1,i_2=r_1}^{r_2}\sum\limits_{j_1,j_2=s_1}^{s_2}[\delta_{i_1 j_1}\delta_{i_2 j_2}-\delta_{i_1 j_2}\delta_{i_2 j_1}]\cdot z_{i_1,2}  z_{i_2,2} z_{j_1,1} z_{j_2,1}\\
		&=& 
		\sum\limits_{i_1,i_2=r_1}^{r_2}\sum\limits_{j_1,j_2=s_1}^{s_2} z_{i_1,2}z_{i_2,2}z_{i_1,1}z_{i_2,1}-z_{i_1,2}z_{i_2,2}z_{i_2,1}z_{i_1,1}\\
		&=& 0\,.
	\end{eqnarray*}
	Suppose that \eqref{vanishing terms} is true for $m-1\,.$ From \eqref{det-product-epsilon}, and by calculating the determinant according to the first line, we can write
	\begin{equation}\label{induction-}
		\varepsilon_{i_1\dots i_m}\varepsilon_{j_1\dots j_m}=\sum_{k=1}^m (-1)^{1+k}\delta_{i_1 j_k}\cdot\varepsilon_{i_2\dots i_m}\varepsilon_{j_1\dots j_{k-1}j_{k+1}\dots j_m}.
	\end{equation}
	Now by inserting \eqref{induction-} into \eqref{vanishing terms} and by a simple change of indices we obtain the result.
\end{proof}

Each element $w$ of the subgroup $\U m\times\U n$ of $\U{m+n}$ can be written as a block matrix
$$w=\begin{bmatrix} 
w_1& 0\\
0 & w_2	
\end{bmatrix},$$
where $w_1\in\U m$ and $w_2\in\U n$.  For each permutation $\pi\in\Pi_{m,n}$ the complex-valued function $\hat\phi_\pi$ satisfies  
\begin{equation}\label{equation-det}
\hat\phi_\pi(z\cdot w)=(\det w_1)\cdot\hat\phi_\pi(z),
\end{equation} 
for all $w\in\U m\times\U n$ and $z\in\U{m+n}$.

\begin{theorem}\label{theorem-complex-grass}
Let $\hat\E_{m,n}=\{\hat\phi_1,\hat\phi_2\,\dots,\hat\phi_{{\bf P}(m+n,m)}\}$ be the above eigenfamily on the unitary group $\U{m+n}$.  If $P,Q:\cn^{{\bf P}(m+n,m)}\to\cn$ are linearily independent homogeneous polynomials of the same positive degree then the quotient
$$\hat\Phi=\frac{P(\hat\phi_1,\dots ,\hat\phi_{{\bf P}(m+n,m)})}{Q(\hat\phi_1,\dots ,\hat\phi_{{\bf P}(m+n,m)})}$$ 
is a non-constant harmonic morphism on the open and dense subset
$$\Z_Q=\{p\in\U{m+n}| \ Q(\hat\phi_1(p),\dots ,\hat\phi_{{\bf P}(m+n,m)}(p))\neq 0\}.$$
This induces a non-constant harmonic morphism $\Phi:\text{\bf pr}(\Z_Q)\to\cn$ on the open and dense subset $\text{\bf pr}(\Z_Q)$ of the complex Grassmannian $$G_m(\cn^{m+n})=\U{m+n}/\U m\times\U n,$$ where $\text{\bf pr}:\U{m+n}\to G_m(\cn^{m+n})$ is the natural projection.
\end{theorem}

\begin{proof}
The fact that $\hat\Phi$ is a harmonic morphism follows immediately from Theorem \ref{theorem-rational}.  This is clearly invariant under the right action of the subgroup $\U m\times\U n$ on $\U {m+n}$ and hence induces the map $\Phi:\text{\bf pr}(\Z_Q)\to\cn$ which also is a harmonic morphism, since the natural projection is a Riemannian submersion with totally geodesic fibres.
\end{proof}

\section{The Quaternionic Unitary Group $\Sp{m+n}$ Revisited}
\label{section-quaternionic-unitary}

In this section we construct eigenfamilies on the compact quaternionic unitary group $\Sp {m+n}$ which is the intersection of the unitary group $\U{2m+2n}$ and the  representation of the quaternionic general linear group $\GLH {m+n}$ in $\cn^{2(m+n)\times 2(m+n)}$ given by
$$
(z+jw)\mapsto q=\begin{bmatrix}
	z_{11}    &  \dots & z_{1,m+n}   & -\bar w_{11}    & \dots  & -\bar w_{1,m+n}\\
	\vdots    & \ddots & \vdots      & \vdots    & \ddots & \vdots  \\
	z_{m+n,1} & \dots  & z_{m+n,m+n} & -\bar w_{m+n,1} & \dots  & -\bar w_{m+n,m+n}\\
	
	w_{11} &  \dots   & w_{1,m+n}   & \bar z_{11}  & \dots  & \bar z_{1,m+n}\\
	\vdots       & \ddots   & \vdots      & \vdots    & \ddots & \vdots  \\
	w_{m+n,1} & \dots & w_{m+n,m+n} & \bar z_{m+n,1} & \dots  & \bar z_{m+n,m+n}\\
\end{bmatrix}.
$$
\medskip

The following result was first stated in Lemma 6.1 of \cite{Gud-Mon-Rat-1}.

\begin{lemma}\label{lemma-Sp}
For the indices $1\le j\le 2(m+n)$ and $1\le\alpha\le (m+n)$ let $q_{j\alpha}:\Sp {m+n}\to\cn$ be the complex valued matrix coefficients of the standard representation of $\Sp {m+n}$. Then the tension field $\hat\tau$ and the conformality operator $\hat\kappa$ satisfy the following relations 
$$
\hat\tau(q_{j\alpha})= -\frac{2(m+n)+1}2\cdot q_{j\alpha}\ \ \text{and}\ \  \hat\kappa(q_{j\alpha},q_{k\beta})=-\frac 12\cdot q_{j\beta}q_{k\alpha}
.$$
\end{lemma}
\smallskip

Let $\Pi_{m,n}$ denote the set of permutations $\pi=(r_1,r_2,\dots ,r_{m})$ such that $1\le r_1<r_2<\dots r_m\le 2(m+n)$ and $\phi_\pi:\Sp{m+n}\to\cn$ be the determinant of the minor associated with $\pi$ i.e. 
$$\hat\phi_\pi(q)=\det
\begin{bmatrix}
	q_{r_1,1}  & \dots  & q_{r_{1},m}\\
	\vdots     & \ddots & \vdots   \\
	q_{r_m,1} & \dots  & q_{r_{m},m}
\end{bmatrix}.$$

\begin{theorem}\label{theorem-main-Spn}
	The set $\hat\E_m=\{\hat\phi_\pi\,|\, \pi\in\Pi_{m,n}\}$ of complex-valued functions is an eigenfamily on the quaternionic unitary group $\Sp{m+n}$ i.e. the tension field $\hat\tau$ and the conformality operator $\hat\kappa$ satisfy 
	$$\hat\tau (\hat\phi)=-\frac{(m^2+m+2\,n+1)}2\cdot\hat\phi\ \ \text{and}\ \ \hat\kappa(\hat\phi,\hat\psi)=-\frac m2\cdot\hat\phi\cdot\hat\psi,$$
	for all $\hat\phi,\hat\psi\in\hat\E_m$.
\end{theorem}

\begin{proof}
	The technique used for proving Theorem \ref{theorem-main-Un} works here as well.
\end{proof}

Each element $q$ of the subgroup $\Sp m\times\Sp n$ of $\Sp{m+n}$ can be written as a block matrix
$$q=\begin{bmatrix} 
	q_1& 0\\
	0 & q_2	
\end{bmatrix},$$
where $q_1\in\Sp m$ and $q_2\in\Sp n$.  For each permutation $\pi\in\Pi_{m,n}$ the complex-valued function $\hat\phi_\pi$ satisfies  
\begin{equation*}\label{equation-det-q}
	\hat\phi_\pi(p\cdot q)=(\det q_1)\cdot\hat\phi_\pi(p)=\hat\phi_\pi(p),
\end{equation*} 
for all $q\in\Sp m\times\Sp n$ and $p\in\Sp{m+n}$.  This means that each element $\hat\phi_\pi\in\hat\E_m$ induces a complex-valued function 
$$\phi_\pi:\Sp{m+n}/\Sp m\times\Sp n\to\cn$$ on the quaternionic Grassmannian.

\begin{theorem}
The set $\E_m=\{\phi_\pi\,|\, \pi\in\Pi_{m,n}\}$ of complex-valued functions is an eigenfamily on the quaternionic Grassmannian 
$$\Sp{m+n}/\Sp m\times\Sp n$$ 
i.e. the tension field $\tau$ and the conformality operator $\kappa$ satisfy 
$$\tau (\phi)=-\frac{(m^2+m+2\,n+1)}2\cdot\phi\ \ \text{and}\ \ \kappa(\phi,\psi)=-\frac m2\cdot\phi\cdot\psi,$$
for all $\phi,\psi\in\E_m$.
\end{theorem}

\begin{proof}
The statement follows directly from the fact that the natural projection $\Sp{m+n}\to\Sp{m+n}/\Sp m\times\Sp n$ is a Riemannian submersion.
\end{proof}

\section{Acknowledgements}

The authors are grateful to Fran Burstall, Adam Lindström, Thomas Munn and Marko Sobak for useful discussions on this work.

The first author would like to thank the Department of Mathematics at Lund University for its great hospitality during her time there as a postdoc.

\section{Declarations}

The authors confirm that the data supporting the findings of this study are available and not applicable.  The authors have no relevant financial or non-financial interests to disclose. The authors declare that no funds, grants or other support were received during the preparation of the manuscript.


\end{document}